\documentclass[reqno]{amsart}
\usepackage{amssymb,amsmath,amsfonts,bm}
\usepackage{dsfont}
\usepackage{epsfig}
\usepackage{graphicx}
\usepackage{esint}
\usepackage{enumerate}
\usepackage{verbatim}
\usepackage{color}
\usepackage{hyperref}
\usepackage[pagewise]{lineno}
\usepackage{caption}
\usepackage{geometry}
\usepackage{tikz}  
\usetikzlibrary{arrows}
\usepackage{rotating}
\usetikzlibrary{shapes,quotes,calc,angles,arrows,through,intersections,shadings}
\usepackage{tkz-euclide}
\tikzstyle{int}=[draw, fill=blue!20, minimum size=2em]
\tikzstyle{init} = [pin edge={to-,thin,black}]


\newcommand{\vertiii}[1]{{\left\vert\kern-0.25ex\left\vert\kern-0.25ex\left\vert #1 
    \right\vert\kern-0.25ex\right\vert\kern-0.25ex\right\vert}}

\date{} 

\def\l{{\langle }}
\def\r{{\rangle }}

\theoremstyle{plain}
\newtheorem{theorem}{Theorem}[section]
\newtheorem{definition}[theorem]{Definition}
\newtheorem{proposition}[theorem]{Proposition}

\newtheorem{lemma}[theorem]{Lemma}

\newtheorem{remark}[theorem]{Remark}

\numberwithin{theorem}{section}
\numberwithin{equation}{section}
\numberwithin{figure}{section}

\let\oldtocsection=\tocsection
 
\let\oldtocsubsection=\tocsubsection
 
\let\oldtocsubsubsection=\tocsubsubsection
 
\renewcommand{\tocsection}[2]{\hspace{0em}\oldtocsection{#1}{#2}}
\renewcommand{\tocsubsection}[2]{\hspace{1em}\oldtocsubsection{#1}{#2}}
\renewcommand{\tocsubsubsection}[2]{\hspace{2em}\oldtocsubsubsection{#1}{#2}}

\begin{document}

\parskip=8pt

\title[Global well-posedness and stability of the inhomogeneous kinetic wave equation]
{Global well-posedness and stability of the inhomogeneous kinetic wave equation near vacuum}
\author[Ioakeim Ampatzoglou]{Ioakeim Ampatzoglou}
\address{Ioakeim Ampatzoglou,  
Baruch College, The City University of New York, USA}
\email{ioakeim.ampatzoglou@baruch.cuny.edu}

\begin{abstract}
In this paper, we prove global in time existence, uniqueness and stability of mild solutions near vacuum for the 4-wave inhomogeneous kinetic wave equation, for Laplacian dispersion relation in dimension $d=2,3$. We also show that for non-negative initial data, the solution remains non-negative. This is achieved by connecting the inhomogeneous kinetic wave equation, for such dimensions, to the cubic part of the  quantum Boltzmann equation for bosons, with Maxwell or hard potential and no collisional averaging. 
\end{abstract}

\maketitle
\section{Introduction}
The problem of understanding the behavior of large systems of nonlinear interacting waves is of fundamental importance in the community of mathematical physics. However, with the size of the system being extremely large, deterministic prediction of its evolution in time is practically impossible and one resorts to  a kinetic description. The kinetic theory of waves, referred to as wave turbulence theory, provides a mesoscopic framework for studying averaging quantities of the system e.g. the point energy spectrum, but still obtaining  a statistically accurate prediction in time. This is in general achieved through the means of an effective equation, which in the case of wave turbulence is the kinetic wave equation (KWE).

This paper focuses on the global in time well-posedness and stability of the 4-wave space inhomogeneous kinetic wave equation for initial data close to vacuum, Laplacian dispersion relation and physical dimension $d=2,3$. The main idea of the paper is that, for such dimensions, it is possible to connect the inhomogeneous kinetic wave equation to the cubic part of the  quantum Boltzmann equation for bosons, with Maxwell or hard potential and no collisional averaging, see Lemma \ref{lemma I} for more details. Since the well-posedness of Boltzmann-type equations near vacuum has been widely studied \cite{KS, IS, hamdache,beto85, to86, beto87, to88,po88, pato89, go97, alonso-gamba, alonso, st10, ours KS,  hanoyu07, hano10, BN} in the past, we employ techniques of the classical kinetic theory for Boltzmann-type equations to address existence, uniqueness and stability  of global in time mild  solutions (see Section \ref{sec: mild solutions} for the precise definition of a mild solution)  to the spatially inhomogeneous (KWE). Up to the author's knowledge, this is the first paper which addresses the global in time well-posedness  of the inhomogeneous kinetic wave equation.

\subsection{The inhomogeneous kinetic wave equation and functional spaces}
We study the global in time well-posedness and stability of the space inhomogeneous kinetic wave equation for 4-wave interactions when the initial data are near vacuum in dimension $d=2,3$. The equation for 4-wave interactions is given by
\begin{equation}\label{KWE}
\begin{cases}
\partial_t f+v\cdot\nabla_x f=\mathcal{C}[f],\\
f(t=0)=f_0,
\end{cases}
\end{equation}
where the collisional kernel is
\begin{equation}\label{collisional kernel}
\mathcal{C}[f](t,x,v)=\int_{\mathbb{R}^{3d}}\delta(\Sigma)\delta(\Omega)ff_1f_2f_3\left(\frac{1}{f}+\frac{1}{f_1}-\frac{1}{f_2}-\frac{1}{f_3}\right)\,dv_1\,dv_2\,dv_3,
\end{equation}
 the resonant manifolds are given by
\begin{align*}
\Sigma&=v+v_1-v_2-v_3,\\
\Omega&=\omega(v)+\omega(v_1)-\omega(v_2)-\omega(v_3),
\end{align*}
$\omega:\mathbb{R}^d\to\mathbb{R}$ is the dispersion relation, and we denote $f:=f(t,x,v)$, $f_i:=f(t,x,v_i)$ for $i\in\{1,2,3\}$. In this paper we consider the classical Laplacian dispersion relation $\omega(v)=|v|^2$, so 
$$\Omega=|v|^2+|v_1|^2-|v_2|^2-|v_3|^2.$$

\noindent The initial data $f_0$ are assumed to be exponential near vacuum. More specifically, given $\alpha,\beta>0$, the initial data will lie in the Banach space of Maxwellian bounded continuous functions:
$$\mathcal{M}_{\alpha,\beta}=\left\{f\in C(\mathbb{R}^d\times\mathbb{R}^d,\mathbb{R}):\sup_{x,v}|f(x,v)|e^{\alpha|x|^2+\beta|v|^2}<\infty\right\},$$
endowed with the norm
$$\|f\|=\sup_{x,v}|f(x,v)|e^{\alpha|x|^2+\beta|v|^2}.$$
Of particular importance will be the set of non-negative initial data:
$$\mathcal{M}_{\alpha,\beta}^+:=\left\{f\in\mathcal{M}_{\alpha,\beta}:\,f(x,v)\geq 0,\quad\forall (x,v)\in\mathbb{R}^d\times\mathbb{R}^d\right\}.$$
The natural space for the mild solutions we will study, is the Banach space
$$\mathcal{S}_{\alpha,\beta}:=\left\{f\in C([0,\infty),\mathcal{M}_{\alpha,\beta}):\,\vertiii{f}<\infty\right\},$$
where the norm is given by
$$ \vertiii{f}:=\sup_{t\geq 0}\|f(t)\|.$$
We will also be writing
$$\mathcal{S}_{\alpha,\beta}^+:=\{f\in\mathcal{S}_{\alpha,\beta}\,:\, f(t)\in\mathcal{M}_{\alpha,\beta}^+,\,\forall\, t\geq 0\}.$$
It is worth noting that the space $\mathcal{S}_{\alpha,\beta}$ continuously embeds in the space $C([0,T],L^1_{x,v})$, since $\mathcal{M}_{\alpha,\beta}$ embeds in $L^1_{x,v}$. Indeed, given $t\geq 0$, we have
\begin{equation}
\|g(t)\|_{L^1_{x,v}}\leq \|g(t)\|\int_{\mathbb{R}^{d}\times\mathbb{R}^{d}}e^{-\alpha|x|^2-\beta|v|^2}\,dx\,dv=
(\alpha\beta)^{-d/2}\pi^{d}\|g(t)\|.
\end{equation}
Hence
\begin{equation}
\|g\|_{C([0,T],L^1_{x,v})}\leq (\alpha\beta)^{-d/2}\pi^{d}\vertiii{g}.
\end{equation}
\subsection{Backround}
The kinetic wave equation was first introduced independently by Peierls \cite{Peierls} who worked on solid state physics, and  Hasselmann \cite{Hasselmann1, Hasselmann2} in his work on water waves. Later, the topic was revived by Zakharov and collaborators \cite{Zakharov1,Zakharov2} who provided a broad framework applying to various Hamiltonian systems satisfying weak nonlinearity, high frequency, phase randomness assumptions. Nowadays, the kinetic theory of waves, known as wave turbulence theory, is fundamental to the study of nonlinear waves, having applications e.g. in plasma theory \cite{Davidson}, oceanography \cite{Janssen,guide} and crystal thermodynamics \cite{Spohn phonons}. For an introduction to this broad research field, see e.g. Nazarenko \cite{Nazarenko}, Newell-Rumpf \cite{Newell}. 

\subsubsection*{The homogeneous kinetic wave equation} The homogeneous 4-wave (KWE), i.e. equation \eqref{KWE} with no spatial dependence,
can be formally derived from the cubic nonlinear Schr\"odinger equation (NLS) with periodic boundary conditions
\begin{equation}\label{cubic NLS torus}
i\partial_t u +\Delta u=\lambda |u|^2u,\quad x\in \mathbb{T}_L^d,
\end{equation}
 and asymptotically describes  the point-energy distribution of the Fourier modes of the solution to \eqref{cubic NLS torus}
for Gaussian initial data
in the weak nonlinearity limit-large box limit $\lambda\to 0$, $L\to\infty$.

The first rigorous result regarding derivation of the homogeneous (KWE) was obtained in the pioneering work of Lukkarinen and Spohn \cite{LS}, who were able to reach the kinetic timescale, which the time that one expects the kinetic behavior of the system to emerge, for the cubic nonlinear Schr\"odinger equation (NLS) at statistical equilibrium, leading to a linearized version of the kinetic wave equation (see also \cite{Faou}). The key idea in \cite{LS} is to use Feynmann diagrams in order to control higher order correlations and has inspired most of the subsequent works.
The derivation for random data out of statistical equilibrium was first addressed by Buckmaster, Germain, Hani, Shatah  \cite{BGHS} using Strichartz estimates to control the error term. However, the derivation was shown to times much smaller than the kinetic timescale.  Later,  Collot and Germain \cite{CG1,CG2}, inspired by the ideas of \cite{LS} (construction of an approximate solution, control of the higher order terms via Feynmann diagramms)  estimated  the error in Bourgain spaces instead of Strichartz spaces and were able reach  the kinetic timescale up to arbitrarily small polynomial loss. At the same time, a similar result was obtained independently by Deng and Hani \cite{DengHani1}. 
 Later, in a pioneering work, Deng and Hani \cite{DengHani2} reached  the kinetic timescale for the cubic (NLS), which provides the first full derivation of the homogeneous (KWE) for (NLS). The same authors addressed propagation of chaos and full range of scaling laws in \cite{deha21,deha23}. Recently, they extended the derivation to longer times in \cite{dehalong}. Under the assumption of multiplicative noise,  Staffilani and Tran \cite{StafTran} reached the kinetic timescale as well for the Zakharov-Kuznetsov  equation. 

Regarding the well-posedness of the homogeneous (KWE), the question of local existence and uniqueness for 4-wave interactions was first addressed in \cite{EV} for velocity isotropic solutions and Laplacian dispersion relations. It is also proved in \cite{EV} that the equation admits global, measure valued, weak solutions, and that condensation can occur.  Existence and uniqueness of radial weak solutions to
a slightly simplified version of the 4-wave kinetic equation for general power-law dispersion has been proved  in \cite{Merino}. For general solutions,  optimal local well-posedness was shown in \cite{GIT}. The results of \cite{GIT} hold in $L^\infty$ for more general dispersion relations, and in $L^2$ for Laplacian dispersion relation. Additionally, stability of solutions in $L^2$ near equilibrium was recently shown in \cite{Menegaki}, while stability and cascades of the Kolmogorov-Zakharov spectrum was shown in \cite{codige22}.
\subsubsection*{The inhomogeneous kinetic wave equation} 
The inhomogeneous (KWE) appears in the physics literature  \cite{Zakharov1,Zakharov2}, where existence of a transport term is physically relevant. In particular, this type of equations are widely used in the prediction of wave propagation in the ocean. Moreover,  Spohn \cite{Spohn phonons} discusses the emergence of an inhomogeneous kinetic wave equation, which he calls phonon  Boltzmann equation, and addresses its connection to nonlinear waves. Therefore, addressing its global well-posedness would be a question of physical interest.

The inhomogeneous 4-wave (KWE) \eqref{KWE}
can be formally derived from the cubic nonlinear Schr\"odinger equation in the whole space  
\begin{equation}\label{cubic NLS free space}
i\partial_t u +\Delta u=\lambda |u|^2u,\quad x\in\mathbb{R}^d,
\end{equation}
by taking the rescaled Wigner transform 
$$
W^\epsilon[u] (t,x,v) = \frac{1}{(2\pi)^{d/2}} \epsilon^{-d} \mathbb{E} \int \overline{u (t,x+\frac z 2)}u(t,x-\frac z 2 ) e^{i\frac{v}{\epsilon}\cdot z}\,dz,
$$
of the solution to \eqref{cubic NLS free space}
for Gaussian initial data exhibiting randomness at a scale $\sim\epsilon$ with an envelope at a scale $\sim 1$. Upon rescaling to the kinetic time, the solution of \eqref{KWE} is asymptotically described by $W^\epsilon[u](t,x,v)$   in the weak nonlinearity-high frequency limit $\lambda,\epsilon\to 0$.
Roughly speaking, the Wigner transform provides a measure of the amount of energy of $u$ (in $L^2$) localized in phase space at position $x$ and frequency $v/\epsilon$. 

Regarding the rigorous derivation of the inhomogeneous (KWE),  the first rigorous result justifying a derivation of a 3-wave inhomogeneous kinetic wave equation from dispersive dynamics was recently obtained by the author in collaboration with Collot and Germain \cite{inhomogeneous}, who derived the inhomogeneous (KWE) up to an arbitrarily small polynomial loss of the kinetic time scale for dispersion relations close to Laplacian and quadratic nonlinearities. In the stochastic setting, where
time-dependent forcing is permitted in the equation, Hannani, Rosenzweig, Staffilani, and Tran \cite{Staf inho}
have made contributions to the study of a KdV-type equation, reaching the kinetic time, while recently Hani, Shatah and Zhu \cite{Hani-Shatah-Zhu} studied inhomogeneous turbulence for Wick NLS.  Up to the author's knowledge, well-posedness of the inhomogeneous (KWE) has not been addressed in the past, and this is the aim of present paper.

\subsection{Connection with the quantum Boltzmann equation}\label{sec:connection} Although the inhomogeneous (KWE) can be derived from the Schr\"odinger equation as described above, it can be seen as a simplified model of the quantum Boltzmann equation for bosons  
\begin{equation}\label{quantum Boltzmann}
\partial_t f+v\cdot\nabla_x f=Q[f],
\end{equation}
which describes the evolution of the probability density of a gas of quantum particles, where both classical collisional effects as well as the Bose-Einstein condensate for bosons at low temperature are taken into consideration.
The collisional operator in \eqref{quantum Boltzmann} is given by
\begin{align}
Q[f]&=\int_{\mathbb{R}^d\times\mathbb{S}^{d-1}}|u|^\gamma b(\hat{u}\cdot\omega)\left(f'f_1'\left(1+f\right)\left(1+f_1\right)-ff_1\left(1+f'\right)\left(1+f_1'\right)\right)\,d\omega\,dv_1,\label{quantum operator}
\end{align}
where $f:=f(t,x,v),\,\,f_1:=f(t,x,v_1),\,\,f':=f(t,x,v'),\,\,f_1':=f(t,x,v_1'),$
 $u:=v_1-v$ denotes the relative velocity of the incoming particles with velocities $v,v_1$, $\omega\in\mathbb{S}^{d-1}$ denotes their unit relative position, $v',v_1'$ are the velocities after the elastic collision given by:
 \begin{equation}
 \begin{aligned}
v'&=v+(\omega\cdot u)\omega,\\
v_1'&=v_1-(\omega\cdot u)\omega. 
 \end{aligned}
 \end{equation}
In \eqref{quantum operator}, $\gamma\in(1-d,1]$ represents the type of potential considered. When $\gamma<0$ the potential is soft, when $\gamma=0$ we have Maxwell molecules, when $0<\gamma<1$ the potential is moderately hard, and when $\gamma=1$ the potential is hard.  Of particular interest to us will be the Maxwell molecules and the hard potentials.
 The function  $b:[-1,1]\to \mathbb{R}$  is a measurable, non-negative, even function which represents the collisional averaging and is referred as the angular cross-section. 
 
 Note that due to cancellations the operator $Q[f]$ is essentially the sum of the classical quadratic Boltzmann operator plus a cubic quantum term, which corresponds to the Bose-Einstein condensate i.e. 
\begin{equation}\label{cl+qu}
Q[f]=Q_{cl}(f,f)+Q_{qu}(f,f,f),
\end{equation}
where
\begin{align}
Q_{cl}(f,f)&=\int_{\mathbb{R}^d\times\mathbb{S}^{d-1}}|u|^\gamma b(\hat{u}\cdot\omega)\left(f'f_1'-ff_1\right)\,d\omega\,dv_1,\label{classical op}\\
Q_{qu}(f,f,f)&=\int_{\mathbb{R}^d\times\mathbb{S}^{d-1}}|u|^\gamma b(\hat{u}\cdot\omega)\left(f'f_1'\left(f+f_1\right)-ff_1\left(f'+f_1'\right)\right)\,d\omega\,dv_1.\label{quantum op}
\end{align}

One reason that physicists study equations of the type \eqref{KWE} rather than the full equation \eqref{quantum Boltzmann} is that  the distribution function f (i.e. the solution
of the quantum Boltzmann equation for bosons) at very low temperature becomes large near
the mean velocity: $ |v| << 1 \Rightarrow f >> 1 $, so the quadratic term $f_2f_3-ff_1$ can be
omitted in comparison with the cubic term $f_2f_3(f+f_1)-ff_1(f_2+f_3)$, see e.g. Section C of \cite{setk97}. For instance, the
equilibrium (i.e. the Bose-Einstein distribution) of the original equation at very low temperature
is large near $v = 0$:
$$\frac{1}{e^{a+b|v|^2}-1}\approx \frac{1}{a+b|v|^2},\quad \text{for }|v|<<1,\quad 0<a<<1,\quad b>0,$$
and $\frac{1}{a+b|v|^2}$ is indeed an equilibrium for \eqref{KWE}. This type of equilibria for the (KWE) are called Rayleigh-Jeans distributions.

\subsection{Statement of the main results}  We now state the main results of this paper. We first prove global in time existence, uniqueness and stability of mild solutions, when the initial data are near vacuum:	
\begin{theorem}\label{theorem general intro}
Let $\alpha,\beta>0$ and $0<R\le\frac{\alpha^{1/4}}{4\sqrt{6}K_{d,\beta}^{1/2}}$, where $K_{d,\beta}>0$ is the constant given in \eqref{K_beta text}. Let $f_0\in\mathcal{M}_{\alpha,\beta}$ with $\|f_0\|\leq R$. Then equation \eqref{KWE} has a unique mild solution $f$ satisfying the bound
\begin{equation}\label{estimate on the solution}
|f(t,x,v)|\leq 2R e^{-\alpha |x-tv|^2-\beta|v|^2},\quad\forall (t,x,v)\in [0,\infty)\times\mathbb{R}^d\times\mathbb{R}^d.
\end{equation}
Additionally, if $f_0,g_0\in\mathcal{M}_{\alpha,\beta}$ with $\|f_0\|, \|g_0\|\leq R$, and $f,g$ are the corresponding mild solutions to \eqref{KWE}, the following stability estimate holds:
\begin{equation}\label{main stability estimate}
|f(t,x,v)-g(t,x,v)|\leq 2\|f_0-g_0\|e^{-\alpha|x-tv|^2-\beta|v|^2}\quad\forall (t,x,v)\in [0,\infty)\times\mathbb{R}^d\times\mathbb{R}^d.
\end{equation}
\end{theorem}

\begin{remark}
The uniqueness claimed holds in the class of solutions of \eqref{KWE} satisfying the bound \eqref{estimate on the solution}.
\end{remark}
When the initial data are non-negative, we show that the corresponding solution of \eqref{KWE} remains non-negative in time:
\begin{theorem} \label{theorem positive intro}
Let $\alpha,\beta>0$ and $0<R\leq \frac{\alpha^{1/4}}{4\sqrt{6}K_{d,\beta}^{1/2}}$, where $K_{d,\beta}>0$ is the constant given in \eqref{K_beta text}. Let $f_0\in\mathcal{M}_{\alpha,\beta}^+$ with $\|f_0\|\leq R$. Then, there exists a unique non-negative mild solution of \eqref{KWE} with
\begin{equation*}
0\leq f(t,x,v)\leq 2R e^{-\alpha|x-tv|^2-\beta|v|^2},\quad\forall (t,x,v)\in [0,\infty)\times\mathbb{R}^d\times\mathbb{R}^d.
\end{equation*}
\end{theorem}
\noindent Theorem \ref{theorem general intro} is proved in Section \ref{sec: gwp}, while Theorem \ref{theorem positive intro} is proved in Section \ref{sec:non-negative solution}.

\begin{remark}
As described in Subsection \ref{sec:connection} a solution $f$ of the inhomogeneous (KWE) \eqref{KWE} approximates a solution of the quantum Boltzmann equation \eqref{quantum Boltzmann} when $f>>1$. However, in this paper the solutions we obtain are small.  In the future, we plan  addressing the well posedness of \eqref{KWE} for initial data which become large near $v=0$, in order for these solutions to be relevant for \eqref{quantum Boltzmann} as well.
\end{remark}

\subsection{Strategy of the proofs} The main idea of the present paper is to connect equation \eqref{KWE} to the cubic part of the quantum Boltzmann equation for bosons with  hard potential and no collisional averaging, and then employ techniques used in the context of kinetic theory of particles.

In particular, we show that \eqref{KWE} is connected to \eqref{quantum operator} follows: the collisional operator \eqref{collisional kernel} is equivalent to the cubic part of \eqref{quantum operator} for $\gamma=d-2$ and constant angular cross-section $b$. Thus for $d=2$, \eqref{KWE} corresponds to Maxwell molecules, while for $d=3$ to hard potentials.

After establishing the above connection in Lemma \ref{lemma I}, and the appropriate a-priori bounds, we prove global well-posedness and stability using the contraction mapping principle. However, to prove existence of a non-negative solution for non-negative initial data, we use a more delicate argument which takes advantage  of the monotonicity properties of the equation. We achieve that  by employing a strong tool from the kinetic theory of particles, namely the Kaniel-Shinbrot iteration, which is an iterative scheme constructing monotone sequences of subsolutions and supersolutions which in turn converge to the solution of the nonlinear equation, as long as an appropriate beginning condition is satisfied. 

The Kaniel-Shinbrot iteration was introduced for the first time by Kaniel and Shinbrot in \cite{KS} for local in time mild solutions to the Boltzmann equation and used by Illner and Shinbrot \cite{IS} to provide global in time mild solutions to the Boltzmann equation for small initial data.  Later, it has been further used in the context of the Boltzmann equation as well as for Boltzmann-type equations such as inelastic Boltzmann equation, Boltzmann-Enskog equation, relativistic Boltzmann equation, binary-ternary Boltzmann equation, gas mixtures, see e.g. \cite{hamdache,beto85, to86, beto87, to88,po88, pato89, go97, alonso-gamba, alonso, st10, ours KS,  hanoyu07, hano10}. Recently, the Kaniel-Shinbrot iteration has been used for the quantum Boltzmann equation for hard spheres in \cite{BN}.

 This is the first paper employing the Kaniel-Shinbrot iteration in the context of wave turbulence and in particular the inhomogeneous kinetic wave equation \eqref{KWE}. We should mention that this technique relies on the the inhomogeneous nature of the problem, i.e. the fact that, for small enough initial data, the transport dominates the collisions under time evolution. Therefore, we would not expect such techniques to apply in the space homogeneous problem.

\subsection*{Acknowledgements} The author acknowledges support from the NSF grant DMS-2206618 and the Simons collaborative grant on Wave Turbulence. The author is also thankful to Pierre Germain for fruitful discussions on the topic.

\section{Gain and loss operators and the notion of a mild solution} \label{sec: mild solutions}
In this section, we first write the kinetic wave equation in gain and loss form and introduce the notion of a mild solution to \eqref{KWE}. 

\subsection{Gain and loss operators} Namely, notice that \eqref{KWE} can be equivalently written as 
\begin{equation}\label{KWE in gain and loss}
\partial_t f+v\cdot\nabla_x f=G(f,f,f,f)-L(f,f,f,f),
\end{equation}
where the generalized gain operator $G$ and loss operator $L$ are given by
\begin{align*}
G(f,g,h,k)(t,x,v)&=\int_{\mathbb{R}^{3d}}\delta(\Sigma)\delta(\Omega)h_2k_3(f+g_1)\,dv_1\,dv_2\,dv_3,\\
L(f,g,h,k)(t,x,v)&=\int_{\mathbb{R}^{3d}}\delta(\Sigma)\delta(\Omega)fg_1(h_2+k_3)\,dv_1\,dv_2\,dv_3.
\end{align*}
We note that the gain and loss operators are increasing with respect to non-negative inputs.
Moreover, the gain operator is linear with respect to $h,k$ and linear with respect to the vector $(f,g)$. Similarly, the loss operator is linear with respect to $f,g$ and linear with respect to  the vector $(h,k)$. In particular, we have the linearity decompositions 
\begin{align}
&G(f,g,h,k)-G(\tilde{f},\tilde{g},\tilde{h},\tilde{k})\nonumber\\
&=G(f,g,h-\tilde{h},k)+G(f,g,\tilde{h},k-\tilde{k})+G(f-\tilde{f},g-\tilde{g},\tilde{h},\tilde{k}),\label{decomposition G}
\end{align}
and
\begin{align}
&L(f,g,h,k)-L(\tilde{f},\tilde{g},\tilde{h},\tilde{k})\nonumber\\
&=L(f-\tilde{f},g,h,k)+L(\tilde{f},g-\tilde{g},h,k)+L(\tilde{f},\tilde{g},h-\tilde{h},k-\tilde{k}).\label{decomposition L}
\end{align}
Finally, we note that the loss term is local with respect to the first input i.e. we can write $$L(f,g,h,k)=fR(g,h,k),$$ where 
\begin{align*}
R(g,h,k)&=\int_{\mathbb{R}^{3d}}\delta(\Sigma)\delta(\Omega)g_1(h_2+k_3)\,dv_1\,dv_2\,dv_3.
\end{align*}
The operator $R$ is also clearly increasing for non-negative inputs.
\subsection{The transport operator} We now define the transport operator, which is a composition of function $g(t,x,v)$ with the Hamiltonian flow, and will be the fundamental operation for constructing mild solutions. Namely, we define $\#:C^0([0,\infty),L^1_{x,v})\to C^0([0,\infty),L^1_{x,v})$ as:
\begin{equation}\label{transport definition}
g^\#(t,x,v):=g(t,x+tv,v).
\end{equation}
This operator is an invertible isometry, since the free flow is measure preserving. Indeed, fixing arbitrary $t\geq 0$, we have
\begin{align*}
\|g^\#(t)\|_{L^1_{x,v}}=\int_{\mathbb{R}^d\times\mathbb{R}^d}|g(t,x+tv,v)|\,dx\,dv=\int_{\mathbb{R}^d\times\mathbb{R}^d}|g(t,x,v)|\,dx\,dv=\|g(t)\|_{L^1_{x,v}},
\end{align*} 
which after taking supremum in time implies that $\#$ is an isometry on $C^0([0,\infty),L^1_{x,v})$ i.e.
\begin{equation}\label{L^1 isometry}
\|g^\#\|_{C^0([0,\infty),L^1_{x,v})}=\|g\|_{C^0([0,\infty),L^1_{x,v})}.
\end{equation}
The inverse operator $-\#:C^0([0,\infty),L^1_{x,v})\to C^0([0,\infty),L^1_{x,v})$ is clearly given by
\begin{equation}\label{inverse transport operator}
g^{-\#}(t,x,v):=g(t,x-tv,v),
\end{equation}
and is an isometry on $C^0([0,\infty),L^1_{x,v})$ as well.
\subsection{Notion of a mild solution}
Consider  a formal solution $f$ of \eqref{KWE}. Using \eqref{KWE in gain and loss}, the chain rule and integrating in time, we obtain 
\begin{equation}\label{mild kinetic wave}
 f^\#(t)=f_0+\int_0^t G^\#(f,f,f,f)(\tau)\,d\tau-\int_0^tL^\#(f,f,f,f)(\tau)\,d\tau,\quad t\ge 0,
\end{equation}
where we denote $L^\#(f,g,h,k):=(L[f,g,h,k])^\#$ and $G^\#(f,g,h,k):=(G[f,g,h,k])^\#$. One can easily verify that $L^\#[f,g,h,k]=f^\#R^\#[g,h,k]$, where $R^\#[g,h,k]:=(R[g,h,k])^\#$. Clearly, the operators $G^\#, L^\#, R^\#$ share the same monotonicity and linearity properties as $G,L,R$ respectively. In particular, there hold the linearity decompositions:
\begin{align}
&G^\#(f_1,g_1,h_1,k_1)-G^\#(f_2,g_2,h_2,k_2)\nonumber\\
&=G^\#(f_1,g_1,h_1-h_2,k_1)+G^\#(f_1,g_1,h_2,k_1-k_2)+G^\#(f_1-f_2,g_1-g_2,h_2,k_2),\label{decomposition G sharp}
\end{align}
and
\begin{align}
&L^\#(f_1,g_1,h_1,k_1)-L^\#(f_2,g_2,h_2,k_2)\nonumber\\
&=L^\#(f_1-f_2,g_1,h_1,k_1)+L^\#(f_2,g_1-g_2,h_1,k_1)+L^\#(f_2,g_2,h_1-h_2,k_1-k_2).\label{decomposition L sharp}
\end{align}

Motivated by \eqref{mild kinetic wave}, we give the definition of a mild solution to \eqref{KWE} with  as follows:
\begin{definition}
Let $\alpha,\beta>0$ and $f_0\in\mathcal{M}_{\alpha,\beta}$. We say that a function $f\in C^0([0,\infty),L^1_{x,v})$ is a mild solution of \eqref{KWE} with initial data $f_0$ if $f^\#\in \mathcal{S}_{\alpha,\beta}$, and the following integral equation holds
\begin{equation}\label{mild solution definition}
f^\#(t)=f_0+\int_0^t G^\#(f,f,f,f)(\tau)\,d\tau-\int_0^tL^\#(f,f,f,f)(\tau)\,d\tau,\quad t\ge 0.
\end{equation}
\end{definition}
\section{A-priori estimates} The goal of this section is to establish the basic global in time a-priori estimates, namely Proposition \ref{prop: a-priori estimates}, which will be of fundamental importance for proving well-posedness and stability for equation \eqref{KWE}. 

We first provide a key computation which connects the wave kinetic kernel \eqref{collisional kernel} with the cubic part of the quantum Boltzmann kernel \eqref{quantum operator} for $d=2,3$, and will allow us to use estimates used in the context of the Boltzmann equation, namely Lemma \ref{lemma traveling Maxwellian} and Lemma \ref{lemma convolution} in order to prove Proposition \ref{prop: a-priori estimates}. 
 
\begin{lemma}\label{lemma I} Let $v,v_1\in\mathbb{R}^d$, and denote 
$$I(v,v_1)=\int_{\mathbb{R}^{2d}}\delta(v+v_1-v_2-v_3)\delta(|v|^2+|v_1|^2-|v_2|^2-|v_3|^2)\,dv_2\,dv_3.$$
Then
$$ I(v,v_1)=\frac{\omega_{d-1}}{2^{d}}|v-v_1|^{d-2},$$
where $\omega_{d-1}$ denotes the area of the $(d-1)$-dimensional unit sphere.
\end{lemma} 
\begin{proof}
 Substituting $v_3=v+v_1-v_2$, we get
$$I(v,v_1)=\int_{\mathbb{R}^d}\delta(|v_2|^2+|v+v_1-v_2|^2-|v|^2-|v_1|^2)\,dv_2.$$
But notice that 
$$|v_2|^2+|v+v_1-v_2|^2-|v|^2-|v_1|^2=2\left(\left|v_2-\alpha\right|^2-R^2\right),\quad \alpha=\frac{v+v_1}{2},\quad R=\frac{|v-v_1|}{2}.$$
So
\begin{align*}
I(v,v_1)&=\int_{\mathbb{R}^d}\delta(2(|v_2-\alpha|^2-R^2))\,dv_2=\frac{\omega_{d-1}}{2}\int_0^\infty r^{d-1}\delta(r^2-R^2)\,dr\\
&=\frac{\omega_{d-1}}{2}\int_0^\infty\frac{\delta(r-R)+\delta(r+R)}{2R}r^{d-1}\,dr=\frac{\omega_{d-1} R^{d-2}}{4}=\frac{\omega_{d-1}}{2^{d}}|v-v_1|^{d-2}.
\end{align*}
\end{proof}

 Now, we present the two estimates which will be useful to us in the proof of Proposition \ref{prop: a-priori estimates} and have been used in the context of Boltzmann-type equations, see e.g. \cite{IS, alonso-gamba, ours KS, BN}. For convenience of the reader, we provide the proofs below.

\noindent The first estimate is on the time integral of a traveling Maxwellian:
\begin{lemma}\label{lemma traveling Maxwellian}Let $x_0,u_0\in\mathbb{R}^d$, with $u_0\neq 0$ and $\alpha>0$. Then,  the following estimate holds
\begin{equation*}
\int_0^\infty e^{-\alpha|x_0+\tau u_0|^2}\,d\tau\leq \sqrt{\pi}\alpha^{-1/2} |u_0|^{-1}.
\end{equation*}
\end{lemma}
\begin{proof}
By triangle inequality, we have
$$\tau|u_0|-|x_0|\leq |x_0+\tau u_0|\Rightarrow e^{-\alpha|x_0+\tau u_0|^2}\leq e^{-\alpha(\tau| u_0|-|x_0|)^2},\quad\forall \tau\geq 0.$$
Therefore integrating in $\tau$, we obtain
\begin{align*}
\int_0^\infty e^{-\alpha|x_0-\tau u_0|^2}\,d\tau&\leq \int_{-\infty}^\infty e^{-\alpha(\tau| u_0|-|x_0|)^2}\,d\tau\leq \alpha^{-1/2}|u_0|^{-1}\int_{-\infty}^\infty e^{-y^2}\,dy\leq \sqrt{\pi}\alpha^{-1/2}|u_0|^{-1},
\end{align*}
and the estimate is proved.
\end{proof} 
\noindent The second estimate is a convolution-type bound:
\begin{lemma}\label{lemma convolution}
Let $q\in(-d,0]$. Then for any $v\in\mathbb{R}^d$ there holds the uniform convolution estimate
$$\int_{\mathbb{R}^d}|v-v_1|^q e^{-\beta|v_1|^2}\,dv_1\leq  \beta^{-d/2}\pi^{d/2}+\frac{\omega_{d-1}}{d+q},$$
where $\omega_{d-1}$ denotes the area of the $(d-1)$-dimensional unit sphere.
\end{lemma}
\begin{proof}
 Since $q\in (-d,0]$, we have
\begin{align}
\int_{\mathbb{R}^d}|v-v_1|^{q} e^{-\beta|v_1|^2}\,dv_1&\leq\int_{|v-v_1|>1}e^{-\beta|v_1|^2}\,dv_1+\int_{|v-v_1|<1}|v-v_1|^{q} \,dv_1\nonumber\\
&\le \beta^{-d/2}\int_{\mathbb{R}^d}e^{-|x|^2}\,dx+\int_{|y|<1}|y|^{q}\,dy\nonumber\\
&= \beta^{-d/2}\left(\int_{-\infty}^{+\infty}e^{-r^2}\,dr\right)^d+\omega_{d-1}\int_0^1 r^{d-1+q}\,dr\nonumber\\
&= \beta^{-d/2}\pi^{d/2}+\frac{\omega_{d-1}}{d+q}.\label{conv binary neg}
\end{align}
\end{proof}
We are now ready to prove the necessary a-priori estimates on the gain and the loss operators:
\begin{proposition}\label{prop: a-priori estimates} Let $f,g,h,k\in C^0([0,\infty),L^1_{x,v})$. Then,  there hold the estimates:
\begin{align}
\vertiii{\int_0^t L^\#(f,g,h,k)(\tau)\,d\tau}&\leq K_{d,\beta}\alpha^{-1/2}\vertiii{f^\#}\cdot \vertiii{g^\#}(\vertiii{h^\#}+\vertiii{k^\#})\label{a-priori estimate loss}\\
\vertiii{\int_0^t G^\#(f,g,h,k)(\tau)\,d\tau}&\leq K_{d,\beta}\alpha^{-1/2}\vertiii{h^\#}\cdot\vertiii{k^\#}(\vertiii{f^\#}+\vertiii{g^\#})\label{a-priori estimate gain},
\end{align}
where 
\begin{equation}\label{K_beta text}
K_{d,\beta}=\frac{\omega_{d-1}\sqrt{\pi}}{2^{d-1}}\left(\beta^{-d/2}\pi^{d/2}+\frac{\omega_{d-1}}{2d-3}\right),
\end{equation}
and $\omega_{d-1}$ denotes the area of the $(d-1)$-unit sphere.
\end{proposition}
\begin{proof}
We estimate the gain first. Notice that by the definition of the norm of $\mathcal{M}_{\alpha,\beta}$ we have $$|f^\#(t,x,v)|\leq e^{-\alpha|x|^2-\beta|v|^2}\vertiii{f^\#},$$
and the same is true for $g^\#,h^\#,k^\#$.

Now on the resonant manifold, there holds 
$v+v_1=v_2+v_3$ 
and $|v|^2+|v_1|^2=|v_2|^2+|v_3|^2$
 which readily implies
\begin{align*}
|v-v_2|^2+|v-v_3|^2&=2|v|^2-2\l v,v_2+v_3\r +|v_2|^2+|v_3|^2\\
&=3|v|^2+|v_1|^2-2\l v,v+v_1\r\\
&=|v-v_1|^2.
\end{align*}
Hence
\begin{align}\label{conservation laws}
|x+\tau(v-v_2)|^2+|x+\tau(v-v_3)|^2&=|x|^2+|x|^2+2\tau\l x,2v-v_2-v_3\r +\tau^2(|v-v_2|^2+|v-v_3|^2)\nonumber\\
&=|x|^2+|x|^2+2\tau\l x,v-v_1\r+\tau^2|v-v_1|^2\nonumber\\
&=|x|^2+|x+\tau(v-v_1)|^2.
\end{align}
Then, using Fubini's theorem, we take
\begin{align}
&\left|\int_0^t G^\#(f,g,h,k)(\tau,x,v)\,d\tau\right|\nonumber\\
&\leq \vertiii{h^\#}\cdot \vertiii{k^\#}(\vertiii{f^\#}+\vertiii{g^\#})\int_{0}^t\int_{\mathbb{R}^{3d}}\delta(\Sigma)\delta(\Omega)e^{-\alpha(|x+\tau(v-v_2)|^2+|x+\tau(v-v_3)|^2}e^{-\beta(|v_2|^2+|v_3|^2)}\nonumber\\
&\hspace{7cm}\times\left(e^{-\alpha |x|^2-\beta|v|^2}+e^{-\alpha|x+t(v-v_1)|^2-\beta|v_1|^2}\right)\,dv_{1,2,3}\,d\tau\nonumber\\
&\leq 2\vertiii{h^\#}\cdot \vertiii{k^\#}(\vertiii{f^\#}+\vertiii{g^\#})e^{-\alpha|x|^2-\beta |v|^2}\int_{\mathbb{R}^{3d}}\delta(\Sigma)\delta(\Omega)e^{-\beta |v_1|^2}\int_0^\infty e^{-\alpha |x+\tau(v-v_1)|^2}\,d\tau\,dv_{1,2,3}\nonumber\\
&\leq 2\sqrt{\pi}\alpha^{-1/2} \vertiii{h^\#}\cdot \vertiii{k^\#}(\vertiii{f^\#}+\vertiii{g^\#})e^{-\alpha|x|^2-\beta |v|^2}\int_{\mathbb{R}^{d}}|v-v_1|^{-1}e^{-\beta |v_1|^2}I(v,v_1)\,dv_1\label{use of traveling Maxwellian}\\
&\leq \frac{\omega_{d-1}\sqrt{\pi}}{2^{d-1}}\alpha^{-1/2}\vertiii{h^\#}\cdot \vertiii{k^\#}(\vertiii{f^\#}+\vertiii{g^\#})e^{-\alpha|x|^2-\beta |v|^2}\int_{\mathbb{R}^{d}}|v-v_1|^{d-3}e^{-\beta |v_1|^2}\,dv_1\label{use of I}\\
&\leq \frac{\omega_{d-1}\sqrt{\pi}}{2^{d-1}}\left(\beta^{-d/2}\pi^{d/2}+\frac{\omega_{d-1}}{2d-3}\right)\alpha^{-1/2}\vertiii{h^\#}\cdot \vertiii{k^\#}(\vertiii{f^\#}+\vertiii{g^\#})e^{-\alpha|x|^2-\beta |v|^2},\label{use of convolution}
\end{align}
where to obtain \eqref{use of traveling Maxwellian} we use Lemma \ref{lemma traveling Maxwellian}, to obtain \eqref{use of I} we use Lemma \ref{lemma I}, and to obtain \eqref{use of convolution} we use Lemma \ref{lemma convolution} 
for $q=d-3\in (-d, 0]$.
Bringing the exponential to the other side and taking supremum over $x,v$ and $t\ge 0$, we obtain estimate \eqref{a-priori estimate gain}.

\noindent The argument for the loss is similar but simpler since it does not require use of the conservation laws \eqref{conservation laws}, so we omit the proof.
\end{proof}

\section{Well-posedness and stability for arbitrary data near vacuum}\label{sec: gwp}
In this section, we will use the a-priori estimates stated in Proposition \ref{prop: a-priori estimates} to prove Theorem \ref{theorem general intro} through a fixed point argument.

\begin{proof}[Proof of Theorem \ref{theorem general intro}]
We will use the contraction mapping principle. Define the closed set 
$$E=\{g\in \mathcal{S}_{\alpha,\beta}\,:\, \vertiii{g}\leq 2R\}\subset \mathcal{S}_{\alpha,\beta}.$$ Recalling the inverse transport operator $-\#$ given in \eqref{inverse transport operator}, we define the operator $\mathcal{T}:E\to E$ by
$$\mathcal{T}g(t)=f_0+\int_0^t G^\#(g^{-\#},g^{-\#},g^{-\#},g^{-\#})(\tau)\,d\tau-\int_0^t L^\#(g^{-\#},g^{-\#},g^{-\#},g^{-\#})(\tau)\,d\tau.$$

We first prove that $\mathcal{T}$ maps into $E$.  Indeed, for $g\in E$, using triangle inequality and Proposition \ref{prop: a-priori estimates}, we have
\begin{align*}
\vertiii{\mathcal{T}g}&\leq \|f_0\|+\vertiii{\int_0^t G^\#(g^{-\#},g^{-\#},g^{-\#},g^{-\#})(\tau)\,d\tau}+\vertiii{\int_0^t L^\#(g^{-\#},g^{-\#},g^{-\#},g^{-\#})(\tau)\,d\tau}\\
&\leq R+4K_{d,\beta}\alpha^{-1/2}\vertiii{g}^3\\
&\leq R+32K_{d,\beta} \alpha^{-1/2}R^3\\
&=(1+32K_{d,\beta}\alpha^{-1/2}R^2)R\\
&< 2R,
\end{align*}
so $\mathcal{T}:E\to E$. Now, by the triangle inequality, the trilinearity decompositions \eqref{decomposition G}-\eqref{decomposition L}, and Proposition \ref{prop: a-priori estimates}, for $h,g\in E$, we obtain
\begin{align}
\vertiii{\mathcal{T}h-\mathcal{T}g}&\leq 4K_{d,\beta}\alpha^{-1/2}(\vertiii{h}^2+\vertiii{h}\cdot \vertiii{g}+\vertiii{g}^2)\vertiii{h-g}\nonumber\\
&\leq 48 K_{d,\beta}\alpha^{-1/2} R^2\vertiii{h-g}\nonumber\\
&\leq \frac{1}{2}\vertiii{h-g}\label{contraction estimate},
\end{align}
thus $\mathcal{T}:E\to E$ is a contraction. By the contraction mapping principle, $\mathcal{T}$ has a unique fixed point $g\in E$. Then $f:=g^{-\#}$ is clearly the unique mild solution to \eqref{KWE}, corresponding to the initial data $f_0$.

To prove \eqref{stability estimate}, let $f,g$ be the solutions corresponding to $f_0, g_0$ respectively. Then, we have
\begin{align}
f^\#(t)-g^\#(t)=f_0-g_0+\int_{0}^t\left( G^\#\left(f,f,f,f\right)(\tau)-  G^\#\left(g,g,g,g\right)(\tau)\right)\,d\tau &\nonumber\\
-\int_{0}^t\left( L^\#\left(f,f,f,f\right)(\tau)-  L^\#\left(g,g,g,g\right)(\tau)\right)\,d\tau&.\label{stability estimate}
\end{align}
Now, by the triangle inequality, and an estimate similar to \eqref{contraction estimate}, we obtain
$$\vertiii{f^\#-g^\#}\leq \|f_0-g_0\|+\frac{1}{2}\vertiii{f^\#-g^\#},$$
thus $\vertiii{f^\#-g^\#}\leq 2\|f_0-g_0\|$,
and \eqref{stability estimate} follows.
\end{proof}

\section{Existence of a non-negative solution}\label{sec:non-negative solution} In this section, we prove existence of a non-negative solution, when the initial data are non-negative. In order to achieve that, instead of using a fixed point argument,  we will take advantage of the gain and loss form of the equation. We will rely on a strong tool from classical kinetic theory, which preserves the monotonicity properties of the equation, namely the Kaniel-Shinbrot iteration. More specifically,  let $f_0\in\mathcal{M}_{\alpha,\beta}^+$ and $(u_0^\#,l_0^\#)\in\mathcal{M}_{\alpha,\beta}^+\times\mathcal{M}_{\alpha,\beta}^+$. For $n\in\mathbb{N}$, consider the initial value problems:

\begin{equation}\label{lower IVP n}
\begin{cases}
\displaystyle\frac{\,d l_n^\#}{\,dt}+l_{n}^\#R^\#(u_{n-1},u_{n-1},u_{n-1})&=G^\#(l_{n-1},l_{n-1},l_{n-1},l_{n-1}),\\
l_n^\#(0)=f_0,
\end{cases}
\end{equation}

\begin{equation}\label{upper IVP n}
\begin{cases}
\displaystyle\frac{\,d u_n^\#}{\,dt}+u_{n}^\#R^\#(l_{n-1},l_{n-1},l_{n-1})&=G^\#(u_{n-1},u_{n-1},u_{n-1},u_{n-1}),\\
u_n^\#(0)=f_0.
\end{cases}
\end{equation}
By basic ODE theory, solutions to \eqref{lower IVP n}-\eqref{upper IVP n} are given inductively by
\begin{align}
l_n^\#(t)&=f_0\exp\left( -\int_0^t R^\#(u_{n-1},u_{n-1},u_{n-1})(\tau)\,d\tau\right)\nonumber\\
&\quad+\int_0^t G^\#(l_{n-1},l_{n-1},l_{n-1},l_{n-1})(\tau)\exp\left(-\int_\tau^tR^\#(u_{n-1},u_{n-1},u_{n-1})(s)\,ds\right)\,d\tau,\label{solution lower IVP n}
\end{align}
and
\begin{align}
u_n^\#(t)&=f_0\exp\left( -\int_0^t R^\#(l_{n-1},l_{n-1},l_{n-1})(\tau)\,d\tau\right)\nonumber\\
&\quad\quad+\int_0^t G^\#(u_{n-1},u_{n-1},u_{n-1},u_{n-1})(\tau)\exp\left(-\int_\tau^tR^\#(l_{n-1},l_{n-1},l_{n-1})(s)\,ds\right)\,d\tau.\label{solution upper IVP n}
\end{align}
\begin{proposition}\label{kaniel-shinbrot prop} Let $\alpha,\beta>0$, $f_0\in\mathcal{M}_{\alpha,\beta}^+$ and $(l_0^\#,u_0^\#)\in\mathcal{M}_{\alpha,\beta}^+\times\mathcal{M}_{\alpha,\beta}^+$. Let $l^\#_1, u_1^\#$ be the corresponding solutions to \eqref{solution lower IVP n}-\eqref{solution upper IVP n} for $n=1$, and assume that the following beginning condition holds for any $t\geq 0$:
\begin{equation}\label{beginning condition}
0\leq l_0^\#\leq l_1^\#(t)\leq u_1^\#(t)\leq u_0^\#\leq\frac{\alpha^{1/4}}{4K_{d,\beta}^{1/2}}e^{-\alpha|x|^2-\beta|v|^2}.
\end{equation}
Then for all $n\in\mathbb{N}$ and $t\geq 0$, we have
\begin{equation}\label{inductive nesting}
0\leq l_0^\#\leq l_1^\#(t)\leq...\leq l_{n-1}^\#(t)\leq l_{n}^\#(t)\leq u_n^\#(t)\leq u_{n-1}^\#(t)\leq...\leq u_1^\#(t)\leq u_0^\#\leq\frac{\alpha^{1/4}}{4K_{d,\beta}^{1/2}}e^{-\alpha|x|^2-\beta|v|^2}.
\end{equation}
Additionally, the sequences $l_n^\#$, $u_n^\#$ pointwise converge to a common limit $f^\#$ such that $f$ is a non-negative mild solution of the \eqref{KWE} with initial data $f_0$.
\end{proposition}

\begin{proof}
We first will prove \eqref{inductive nesting} inductively. For $n=1$, \eqref{inductive nesting} holds due to the beginning condition \eqref{beginning condition}. Assume \eqref{inductive nesting} holds for $n$, we will show it also holds for $n+1$. It suffices to show $$l_n^\#(t)\leq l_{n+1}^\#(t)\leq u_{n+1}^\#(t)\leq u_n^\#(t).$$
By the induction's assumption, and the monotonicity properties of $R^\#,G^\#$, we have $$R^\#(u_{n},u_{n},u_{n})\leq R^\#(u_{n-1},u_{n-1},u_{n-1}),\quad G^\#(l_{n-1},l_{n-1},l_{n-1},l_{n-1})\leq G^\#(l_n,l_n,l_n,l_n),$$
so the solution's formula \eqref{solution lower IVP n} implies $l^\#_n\leq l_{n+1}^\#$. Similarly, 
$$R^\#(l_{n-1},l_{n-1},l_{n-1})\leq R^\#(l_{n},l_{n},l_{n}),\quad G^\#(u_{n},u_{n},u_{n},u_{n})\leq G^\#(u_{n-1},u_{n-1},u_{n-1},u_{n-1}),$$
so \eqref{solution upper IVP n} implies $u_{n+1}^\#\leq u_n^\#$, while
$$R^\#(l_{n},l_{n},l_{n})\leq R^\#(u_{n},u_{n},u_{n}),\quad G^\#(l_{n},l_{n},l_{n},l_{n})\leq G^\#(u_{n},u_{n},u_{n},u_{n}),$$
thus \eqref{solution lower IVP n}-\eqref{solution upper IVP n} imply $l_{n+1}^\#\leq u_{n+1}^\#$. Hence, \eqref{inductive nesting} follows by induction.

Now, fixing $t\geq 0$. By \eqref{inductive nesting}, the sequence $(l_n^\#(t))_n$ is increasing and upper bounded so $l_n^\#(t)\nearrow l^\#(t)$, while the sequence  $(u_n^\#(t))_n$ is decreasing and lower bounded so $u_n^\#(t)\searrow u^\#(t)$. Moreover, by \eqref{inductive nesting} we have
$$0\leq l^\#(t)\leq u^\#(t)\leq u_0^\#.$$
 Integrating \eqref{solution lower IVP n}-\eqref{solution upper IVP n} in time and using the dominated convergence theorem to let $n\to\infty$, we obtain
\begin{align}
l^\#(t)+ \int_0^t l^\#(\tau)R^\#(u,u,u)(\tau)\,\tau&=f_0+\int_0^t G^\#(l,l,l,l)(\tau)\,d\tau,\label{integrated loss}\\
u^\#(t)+ \int_0^t u^\#(\tau)R^\#(l,l,l)(\tau)\,\tau&=f_0+\int_0^t G^\#(u,u,u,u)(\tau)\,d\tau.\label{integrated gain}
\end{align}
Subtracting \eqref{integrated loss}-\eqref{integrated gain}, using the facts that $l^\#R^\#(u,u,u)=L^\#(l,u,u,u)$ and $u^\#R^\#(l,l,l)=L^\#(u,l,l,l)$, and the triangle inequality, we obtain 
\begin{align}
&|u^\#(t)-l^\#(t)|\nonumber\\
&\leq\int_0^t \left|G^\#(u,u,u,u)(\tau)-G^\#(l,l,l,l)(\tau)\right|\,d\tau+\int_0^t\left|L^\#(l,u,u,u)(\tau)-L^\#(u,l,l,l)(\tau)\right|\,d\tau.\label{estimate on difference of solutions}
\end{align}
By the trilinearity decomposition \eqref{decomposition L}-\eqref{decomposition L} of $L^\#$ and $G^\#$, we have the expansion
\begin{align*}
G^\#(u,u,u,u)-G^\#(l,l,l,l)&=G^\#(u,u,u,u-l)+G^\#(u,u,u-l,l)+G^\#(u-l,u-l,l,l),\\
L^\#(l,u,u,u)-L^\#(u,l,l,l)&=L^\#(l-u,u,u,u)+L^\#(u,u-l,u,u)+L^\#(u,l,u-l,u-l).
\end{align*}
Then, \eqref{estimate on difference of solutions}, triangle inequality, \eqref{prop: a-priori estimates}, and inequality of \eqref{beginning condition} imply
\begin{align*}
\vertiii{u^\#-l^\#}&\leq K_{d,\beta}\alpha^{-1/2}\left(6\vertiii{u^\#}^2+4\vertiii{u^\#}\cdot\vertiii{l^\#}+ 2\vertiii{l^\#}^2 \right)\vertiii{u^\#-l^\#}\\
&\leq K_{d,\beta}\alpha^{-1/2}\|u_0^\#\|^2\vertiii{u^\#-l^\#}\\
&\leq\frac{3}{4}\vertiii{u^\#-l^\#},
\end{align*} 
thus $u=l$. Clearly $f:=u=l$,  is a mild solution of \eqref{KWE}.
\end{proof}

\noindent Now, with the aid of Proposition \ref{kaniel-shinbrot prop}, we will prove Theorem \ref{theorem positive intro}:

\begin{proof}[Proof of Theorem \ref{theorem positive intro}]
 To prove existence, we aim to use Proposition \ref{kaniel-shinbrot prop} for an appropriate choice of $(l_0^\#,u_0^\#)\in\mathcal{M}_{\alpha,\beta}^+\times\mathcal{M}_{\alpha,\beta}^+$. Namely, we define $l_0^\#=0$ and $u_0^\#=Ce^{-\alpha|x|^2-\beta|v|^2}$, 
where
\begin{equation}\label{output constant}
C=2R(1-\sqrt{1-R^{-1}\|f_0\|})\geq 0,
\end{equation}
which is well-defined since $\|f_0\|\leq R$. Moreover, notice that $C$ satisfies the equation
\begin{equation}\label{equation for C}
\|f_0\|+\frac{1}{2}K_{d,\beta}^{1/2}\alpha^{-1/4}C^2=C,
\end{equation} and is estimated as follows:
\begin{equation}\label{bound for C}
C=2R(1-\sqrt{1-R^{-1}\|f_0\|})=\frac{2\|f_0\|}{1+\sqrt{1-R^{-1}\|f_0\|}}\leq 2\|f_0\|\leq 2R.
\end{equation}
 Solving \eqref{solution lower IVP n}-\eqref{solution upper IVP n} for $l_1^\#, u_1^\#$, we obtain
\begin{align*}
l_1^\#(t)&=f_0\exp\left(-\int_0^t R^\#(u_0,u_0,u_0)(\tau)\,d\tau\right),\\
u_1^\#(t)&=f_0+\int_0^t G^\#(u_0,u_0,u_0,u_0)(\tau)\,d\tau.
\end{align*}
We clearly have $0=l_0^\#\leq l_1^\#(t)\leq u_1^\#(t)$. Moreover, by \eqref{bound for C} we have
\begin{align*}
\|u_0^\#\|=C\leq 2R<\frac{\alpha^{1/4}}{4K_{d,\beta}^{1/2}}.
\end{align*}
 Therefore, in order to apply Proposition \ref{kaniel-shinbrot prop}, it suffices to show that $u_1^\#(t)\leq u_0^\#$ for \eqref{beginning condition} to hold. Indeed, by Lemma \ref{prop: a-priori estimates}, bound \eqref{a-priori estimate gain}, and equation \eqref{equation for C}, we have
\begin{align*}
u_1^\#(t)&\leq e^{-\alpha|x|^2-\beta|v|^2}\left(\|f_0\|+2K_{d,\beta}\alpha^{-1/2}\|u_0^\#\|^3\right)\\
&\leq e^{-\alpha|x|^2-\beta|v|^2}\left(\|f_0\|+\frac{1}{2}K_{d,\beta}^{1/2}\alpha^{-1/4}\|u_0^\#\|^2\right)\\
&=e^{-\alpha|x|^2-\beta|v|^2}\left(\|f_0\|+\frac{1}{2}K_{d,\beta}^{1/2}\alpha^{-1/4}C^2\right)\\
&=Ce^{-\alpha|x|^2-\beta|v|^2}\\
&=u_0^\#,
\end{align*}
so \eqref{beginning condition} is satisfied for this choice $(l_0^\#,u_0^\#)$. Thus, existence of a non-negative mild solution to \eqref{KWE} is guaranteed by Proposition \ref{kaniel-shinbrot prop}. Moreover, by \eqref{bound for C}, there holds the bound
$$f^\#(t)\leq u_0^\#=Ce^{-\alpha|x|^2-\beta|v|^2}\leq 2R e^{-\alpha|x|^2-\beta|v|^2}.$$ 
Thus, existence of such a solution is proved. Uniqueness follows immediately from Theorem \ref{theorem general intro}.

\end{proof}

\end{document}